\documentclass[letterpaper, 10 pt, journal]{IEEEtran}

\usepackage{graphicx}
\graphicspath{ {images/} }
 \usepackage{color}
 \usepackage{graphicx}
\usepackage{caption}
\usepackage{subcaption}
\usepackage{amsmath}
\usepackage[utf8]{inputenc}
\usepackage[english]{babel}
\usepackage{amsthm}
 \usepackage{enumitem}
\usepackage{graphics} 
\usepackage{epsfig} 
\usepackage{mathptmx} 
\usepackage{times} 
\usepackage{amsmath} 
\usepackage{amssymb}  
\DeclareMathOperator{\E}{\mathbb{E}}
\DeclareMathOperator*{\argmax}{arg\,max}
 \usepackage{comment}
\newtheorem{theorem}{Lemma}

\newtheorem{cor}{Corollary}
\newtheorem{asum}{Assumption}
\newtheorem{remark}{Remark}

\newtheorem{example}{Example}
\newtheorem{lem}{Lemma}
\newtheorem{pro}{Proposition}
\newtheorem{thm}{Theorem}

\begin{document}
\title{A Contract Design Approach for Phantom Demand Response}

\author{Donya~Ghavidel Dobakhshari\thanks{The authors are with the Department of Electrical Engineering, University of Notre Dame, IN, USA. Email:  (\texttt{dghavide, vgupta2)@nd.edu.} A preliminary version of these results were presented at American Control Conference, 2016~\cite{don}. As compared to that paper, this manuscript considers the case when multiple customers are present. The work was supported  in part by NSF grant... } and         Vijay~Gupta }

\maketitle



%
\IEEEpeerreviewmaketitle

\begin{abstract}
We design an optimal contract between a demand response aggregator (DRA) and power grid customers for incentive-based demand response. We consider a setting in which  the customers are asked to reduce their electricity consumption by the DRA and they are compensated for this demand curtailment. However, given that the DRA must supply every customer with as much power as she desires, a strategic customer can temporarily increase her base load in order to report a larger reduction as  part of the demand response event. The DRA wishes to incentivize the customers  both to  make  costly effort to reduce load and to not falsify the reported load reduction. We model this problem  as a contract design problem and present a solution. The proposed contract consists of two parts: a part that depends on (the possibly inflated) load reduction as measured by the DRA and another that provides a share of the  profit that accrues to the DRA through the demand response event to the customers. Since this profit accrues due to the total load reduction because of the actions taken by all the customers, the interaction among the customers also needs to be carefully included in the contract design.  The contract design and its properties are presented and illustrated through examples. \end{abstract}


\section{Introduction}
Demand Response (DR), in which a utility company or an aggregator motivates customers to curtail their power usage, has now become an acceptable method in situations where high peaks in demand occur, transmission congestion increases, or some power plants are not available to generate enough power~\cite{imp4, va, den, imp6}. Per the  Federal Energy Regulatory Commission (FERC), demand response is the change in electric usage by end-use customers from their normal consumption patterns in response to changes in the price of electricity or any other incentive~\cite{balijepalli2011review}.

In general, DR programs may be divided into two main categories: Price Based Programs (PBP) and Incentive Based Programs (IBP).  PBPs refer to schemes in which the electricity price varies as a function of variables such as the  time of usage or the total demand, with the expectation that the consumers will adjust their demand in response to such a price profile. On the other hand, IBPs offer a constant price for power to every user; however, customers are offered a reward if they reduce their demand when the utility company desires. Classically, these incentives were proposed to be constant and based only on customer participation in the program; however, market-based incentives that offer a reward that varies with the amount of load reduction that a customer achieves have also been proposed. There exists a rich literature for IBPs (e.g, see~\cite{mo,sam,ro, wa} and the references therein) studying the design of suitable incentives with aims such as social welfare maximization, minimization of electricity generation and delivery costs, and reducing renewable energy supply uncertainty for demand response.

In this paper, we consider an incentive based program for demand response where the  customers are rewarded financially by a demand response aggregator (which role can also be filled by a utility company) for their load reduction during DR events. When called upon to reduce their loads, each customer puts in some effort to achieve a true value of load reduction. The effort is costly to the customers since it causes them discomfort. Further,   the amount of effort expended is private knowledge for each customer. Incentivizing customers to put in effort in this setting  is the problem of \textit{moral hazard}~\cite[Chapter~4]{laf}. Following the rich literature back to Holmstrom~\cite{holmstrom1979moral}, as a means to incentivize the customer to put in ample effort in the presence of moral hazard, the  demand response aggregator (DRA) must pay each customer proportional to the effort that the customer puts in.  However, by taking advantage of the fact that the DRA must supply as much power as the customer desires and by anticipating the demand response call, a strategic customer can artificially inflate her base load before an expected DR event. In other words, the true amount of load reduction is also private knowledge for the customer. By artificially inflating the base load, for the same \textit{nominal} load reduction,  the customer can report more \textit{measured} load reduction and gain more financial reward from the DRA~\cite{chao20, chao}.  This implies that the problem of \textit{adverse selection}~\cite[Chapter~3]{laf} is also present. 

That such strategic behavior by customers to exploit IBPs is not idle speculation has been pointed out multiple times~\cite{chao}, \cite{wolak}.  In 2013, it was revealed that the Federal Energy Regulatory Commission issued large civil penalties to customers for exactly this sort of strategic behavior~\cite{FERC2013}.  For instance, Enerwise paid a civil penalty of \$780,000 for wrongly claiming on the behalf of its client, the Maryland Stadium Authority (MSA), that it reduced the baseline electricity usage in 2009 and 2010 at Camden Yards. It may also be pointed out that the possibility of behavior in which ``phantom DR occurs through inflated baseline''   to obtain ``payments for fictitious reductions'' was pointed out in a related but different context  by California ISO in its opinion on FERC order 745~\cite{caliso}. To avoid this {\em phantom demand response}, a payment structure to incentivize a rational customer to provide maximal effort and low (or no) misreporting is needed. 

While there is much literature that uses competitive game theory in smart grids particularly for solutions based on concepts such as pricing (e.g., 
\cite{saad},~\cite{fadlu} and the references therein), much of this literature assumes the users to be truthful and non-anticipatory. While this is often a good assumption in cases where users are price taking and either unable or unwilling to transmit false information, it can lead to overly optimistic results in the framework discussed above. We consider anticipatory and strategic customers that maximize their own profit by predicting the impact of their actions and possibly falsifying any information they transmit. 

Of more interest to our setting is the literature on contract design for DR with information asymmetry and strategic behavior. For instance,~\cite{chao20}  proposed a DR contract that avoids inefficiencies in the presence of a strategic sensor; however, the possibility of baseline inflation was not considered.~\cite{nguy} proposed a DR market to maximize the social welfare; however, the baseline consumption levels were assumed to be known. The works closest to ours in this stream are~\cite{chen20} and~\cite{pra}. \cite{chen20} considered a two-stage game for DR.  Assuming knowledge of the utility function of the  consumer, the authors proposed using a linear penalty  function for the deviation of the usage level from the reported baseline to induce users to report their true baselines, while at the same time adjusting the electricity price appropriately to realize the desired load reduction.~\cite{pra} designed a two-stage mechanism to induce truth telling by the customer irrespective of the utility function of the DRA. The proposed mechanism relied on assuming a linear utility function for the customers, a deterministic  baseline and  a low probability of occurrence of the DR event.  Unlike these works, we design a contract which maximizes the utility function of the DRA (which includes the payment to the customer) and a more general utility function for the customer that includes falsification and effort costs, as well as constraints of individual rationality.


In economics, contract design with either moral hazard or adverse selection alone has a vast literature (for a summary, see, e.g.,~\cite[Chapter~14B]{mas1} and~\cite[Chapter~14C]{mas1}). In the problem we consider, moral hazard followed by adverse selection arises. This combination  is much less discussed in the literature and is significantly more difficult since incentives to solve moral hazard (for instance, through payments that are an increasing function of the reported effort) may, in fact, exacerbate the problem of adverse selection by incentivizing larger falsification of the reported effort. A notable exception is~\cite[Chapter~7]{laf} which considers a specific buyer-seller framework with two hidden actions and two hidden pieces of information. In this stream, the closest  works to our setup  are~\cite{imp3, imp2} who study  the problem of incentivizing a single manager (a single customer in our framework) by the owner of a firm (the DRA in our framework) and propose a contract by assuming accurate revelation of the private information of the manager to the owner in long run. Our formulation includes the more general case of multiple customers with the DRA obtaining  non-accurate knowledge of the load reduction by the customers even in the long run.


The chief contribution of this paper is the design of a contract to  maximize the utility function  of the DRA while incentivizing rational customers to expend costly effort to reduce their load. The contract addresses the issue of moral hazard followed by adverse selection that is enabled by the fact that knowledge of the effort put in as well as that of the true load savings realized are both private to the customers. The contract that we propose consists of two parts: one part that pays the customer based on the (possibly falsified) reported load reduction, and another that provides a share of the profit that accrues to the DRA through the demand response event to the customer. One interesting result is that the optimal contract may lead to both \textit{under-reporting} and  \textit{over-reporting} of load reduction by the customer depending on the true load reduction realized by the customer. In other words, if a strategic customer wishes to maximize her profit, she may sometimes decrease her base load before the DR event to under-report her power reduction as a part of DR event. We also  show that the DRA can realize any arbitrary  demand reduction by contracting with an appropriate number of customers.

The rest of the paper is organized as follows. In Section \ref{sec1}, the problem statement is presented. In Section \ref{sec4}, we propose   a contract structure for the DR problem. Next, in Section \ref{sec3}, we  derive the optimal strategy chosen by DRA and the customers  in response,   discuss the interactions among customers, and  study several extensions of the problem.  In Section \ref{illus}, numerical examples are  provided to illustrate the results.  Section \ref{concl}   concludes the paper and presents some avenues for future work.

\paragraph*{Notation}
$f_{X|A}(x|a)$  (which is often simplified to $f(x|a)$ when the meaning is clear from the context) denotes the probability distribution function (pdf) of random variable $X$ given the event $A=a$. A Gaussian distribution is denoted by  $\mathcal N(m,\sigma^2 )$ where $m$ is the mean and $\sigma $ is the standard deviation. For two functions $g$ and $h$, $g*h$ denotes the convolution between $g$ and $h$.  $\E_X[f]$ specifies that the expectation of function $f$ is taken with respect to the random variable $X$; when $X$ is clear from the context, we abbreviate the notation to $\E[f]$. Given $N$ variables $x_1, \cdots, x_N$, the set defining their collection is denoted by $\{x_i\}_{i=1}^{N}$, or sometimes simply by $\{x_i\}$.


\section{Problem Statement}
\label{sec1}
During a DR event, the DRA calls on the customers to decrease their power consumption. A contract that pays the customers merely for the act of reducing the load  will not incentivize the customers to exert the maximal effort for reducing the load by as much amount as possible. To solve this problem, the DRA may offer a contract that makes the payment to the customer proportional to the load reduction. However, with such a contract, a strategic customer will try to anticipate the DR event and increase her base load, i.e., the load before the demand response event began. This pre-increase allows the customer to reduce the load during the DR event by a larger amount than would have been possible in the absence of such an increase; thus, receiving a larger payment even though the DRA accrues the benefit of only a smaller true load reduction. The central problem considered in this paper is to design a contract that is free from both these problems.

\begin{remark} It is worth pointing out that the falsification of the load reduction claimed by the  customer may happen even though the load at the customer is being monitored constantly and accurately. Further, the DRA can not find the `true' base load by considering the load used by a customer at some arbitrary time before the DR event. For one, this simply shifts the problem of customer manipulation of the load to an earlier time. Second, some of the increase in the base load may be due to true shifts in customer need due to, e.g., increased temperature.
\end{remark}

\subsection{Timeline}
\begin{figure}[tb]
\centering 
\includegraphics[width=9cm, height=4.5cm]{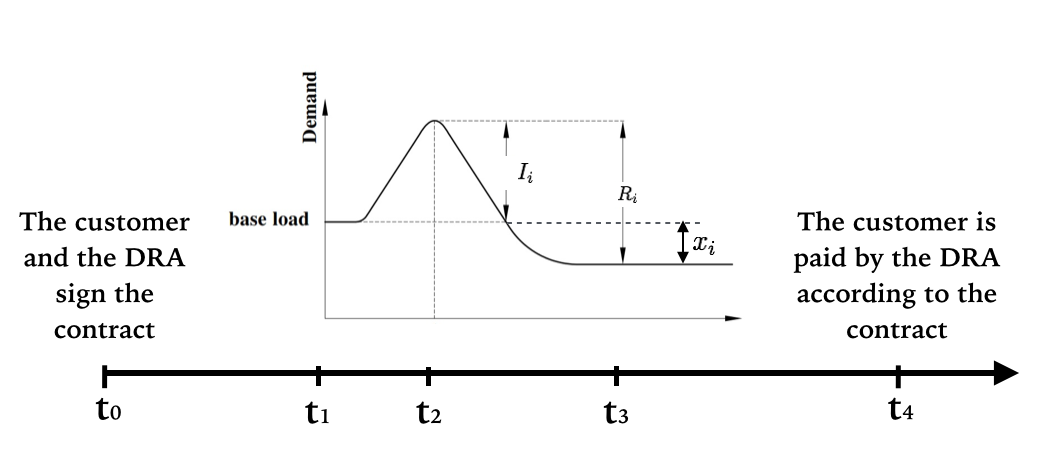}
\caption{Timeline of the DR event and the proposed contract.}
\label{pic_2}
\end{figure}
Consider $N$ customers denoted by $i=1, \cdots , N$ that are contracted with a DRA. We refer to the timeline shown in Figure \ref{pic_2} to explain the sequence of events. At time $t_1$, strategic customers anticipate that a DR event is likely to begin at time $t_{2}$. Accordingly, at this time, each customer $i$ calculates the effort $a_{i}$ she is willing to put in for the load reduction during the DR event. We assume that this effort costs the customer $h(a_i)$. Further, this effort leads to a reduction in the load by an amount $x_i$ that can depend on local conditions that are private knowledge for the $i$-th customer. For instance, a factory might be able to induce a large load reduction with a small effort based on its assembly line requirements given the orders it has to fulfill.  
The DRA is assumed to know the probability density function $f(x_i|a_i)$ according to which $x_i$ is realized, while the customer knows the local conditions and can calculate the value of $x_i$ that will be realized. After this calculation, each customer $i$ at time $t_{1}$ may increase (or decrease) the load by an amount $I_i$ in anticipation of the DR event. 

\begin{asum}
\label{assum11}
The random variables describing the load reductions are conditionally independent given the actions taken by all the customers, so that 
\[
f(x_1,x_2,\cdots, x_N|a_1,a_2, \cdots, a_N)=f(x_1|a_1)f(x_2|a_2)\cdots f(x_N|a_N).\]
\end{asum}
\begin{asum}
\label{asum3}
For ease of computation, we assume that $x_i$ is a noisy signal of the action, i.e.,
  \begin{equation}
  x_i=a_i+e_i,
  \label{ae}
  \end{equation}
where the random variables $\{e_i\}$ are i.i.d  with mean $m_e$ and variance  $\sigma^2$. Assumption~\ref{assum11} can thus be stated as
\[
f(e_1,e_2,\cdots, e_N)=f(e_1)f(e_2)\cdots f(e_N).\]
We first consider the case when $m_e=0$  and extend  the results  to the case when $m_e\neq0$  in Section \ref{ex1}. 
\end{asum}

At time $t_2$, the DR event begins and the DRA calls on the customers to decrease their loads. Each customer $i$ now makes the  predetermined effort $a_i$ leading to a reduction of her load by $x_i$. The DR  event ends at $t_3$ with each customer $i$ having reported that she decreased the load by an amount $R_i$. Note that the true reduction in the load for the $i$-th customer is  $x_i= R_i-I_i$, while the false report\footnote{We wish to emphasize again that the load at the customer  is being accurately monitored at all times.} is $R_i$. 

We also show the times $t_0$ and $t_4$ in the timeline in Figure \ref{pic_2}. At time $t_0$ (much before $t_{1}$), the contract specifying the payment structure is signed between the DRA and the customers. We assume that $t_{0}$ is sufficiently early, so that at $t_{0}$, the customers too do not know the local conditions and must consider their expected utility according to the probability density functions $f(x_i|a_i)$ (or, equivalently, $f(e_i)$). At time $t_4$, at least a part of the payment $P_i$ as specified by the contract to the customers is paid by the DRA to incentivize them to participate in the DR event. The time $t_{4}$ is sufficiently close to the DR event, so that the realized value of $x_i$ is not known at time  $t_4$ to the DRA.  The contract may specify that the rest of the payment is done at some later time $t_{5}$, when the DRA may have more knowledge of the true value of $x_i$. 

\subsection{Utility Functions}
The effort cost suffered by the $i$-th customer for an effort $a_i$ is given by a function $h(a_{i})$ that is known to all the customers and the DRA. Further, for a true reduction $x_i$, if the customer manipulates her base load and reports the reduction to be $R_{i},$ she suffers a falsification cost $g_{i}(R_i-x_i)$. This can model, e.g.,  any extra payment by the customer for boosting her consumption as she manipulated the load prior to the DR event. For simplicity, we assume that $g_{i}(R_{i}-x_{i}) = \beta_{i}\frac{(R_i-x_i)^2}{2},$ $\forall i,$ where $\beta_{i}>0$. Thus, with a payment $P_i$, the utility of the $i$-th customer is given by 
\[
V_i=-h(a_i)-\beta_{i}\frac{(R_i-x_i)^2}{2}+P_i,\qquad i=1,\cdots, N.
\]

The utility of the DRA is given by its net profit, which is the difference of the gross profit that occurs due to the load reduction by the customers and the payments made to the customers as part of the contract. For simplicity, we assume that the gross profit made due to reduction of load $x_i$ is equal to  $x_i$; more complicated cases can be easily considered. With a total payment $\sum_{i=1}^{N} P_i$, the utility of the DRA is given by 
\begin{equation*}
\Pi=\sum_{i=1}^{N} (x_i-P_i).
\end{equation*} 

\subsection{Problem Formulation}
We assume that the customers and the DRA are rational and risk neutral, so that they seek to maximize the expected value of their utility functions. The problem we consider in this paper is for the DRA to design a contract  that maximizes its own utility when rational customers choose \textit{actions} $\{a_i\}$ and \textit{reports} $\{R_i\}$ to optimize their own utility functions. Denote by $\mathcal X$ the set of random variables describing the actual load reductions generated by the customers, i.e,  $\mathcal X \triangleq\{X_1, \cdots,X_{N}\}$ and by $\mathcal{E}$ the set $\{E_1, \cdots,E_{N}\}$.  Further, denote by $\mathcal X_{-i}$ the set of random variables describing the load reductions of all customers except $i$, i.e, $\mathcal X_{-i}\triangleq \mathcal X\backslash \{X_i\}$ and by $\mathcal{E}_{-i}$ the set $\mathcal{E}_\mathcal E \backslash \{E_i\}$. Thus, the optimization problem $\mathcal{P}_{1}$ to be solved by DRA is given by
 \begin{equation*}
\textrm{$\mathcal{P}_{1}$:}
\begin{cases}
&\underset{\{P_i\}}\max  \E_\mathcal E[\Pi]\\
s.t. &\textrm{$\{a_{i}, R_{i}\}$ is chosen to maximize $\E_\mathcal E[V_i]$ by each}\\ &\textrm{ customer $i$}\\
&\textrm{individual rationality and incentive}\\&\textrm{ compatibility constraints for all the customers}.
\end{cases}
\end{equation*}
 
As stated in problem $\mathcal{P}_1$, we impose two constraints on the contract.
\paragraph{Individual rationality}
We assume that the DRA can not force customers to participate in the load reduction program due to political or social reasons. Instead, the contract should be individually rational so that a rational customer chooses to participate. We impose this constraint in the form of ex ante individual rationality.  This constraint requires that no customer chooses to walk away from the contract at time $t_0$ before she knows either  her own load saving or the savings of the other  customers; thus,  $\E_\mathcal E [V_i]\geq 0, \forall{i}.$
\paragraph{Incentive compatibility}
 Incentive Compatibility is a standard constraint imposed in mechanism design which is used to limit the space of the contracts we need to optimize over (see, e.g.,~\cite{myerson1979incentive}).   Specifically, this constraint implies that the utility of the consumers  does not increase if they calculate their report $R_i$ based on any arbitrary quantity other than the true value of their load reduction $x_i$. Further, this constraint also implies that without loss of generality, a customer with private information of load reduction  $x_i$ would always prefer the payment $P_{i}(x_{i})$ over the alternatives $P_{i}(\hat{x}_i)$ for any $\hat{x_i}\neq x_i$. 
 
 We make the following further assumptions: 
 \begin{asum}
 \label{asum4}
 \begin{enumerate}[label=(\roman*)]
\item \textbf{(Deterministic Policies)} The customers choose effort $a_i$ according to deterministic policies. Stochastic  policies would imply additional stochasticity in $\mathcal{P}_{1}$ that we do not consider in this paper.

\item \textbf{(Communication Structure)} Individual customers cannot  communicate with each other, so that the load reduction $R_i$ claimed by the $i$-th customer as well as the true profit $x_i$ and hidden action $a_i$ for this customer are not known to the other customers. The DRA does not have access to $x_i$ and $a_i$ till possibly at a much later time $t_{5}\gg t_{4}.$ 

\item \textbf{(Public Knowledge of Functional Forms)} The functional forms of $h_i$, $g_i$,  the probability distribution functions  $\{f(e_i)\}$,  the weights $\{\beta_i\}$, and the contracts offered  are known to all the customers and the DRA.
\end{enumerate}
\end{asum}

We now proceed to present our solution to the problem $\mathcal{P}_{1}$.

\section{Structure of the Proposed Contract}
\label{sec4}
In this section, we present a contract as a solution of the problem $\mathcal{P}_{1}$. To this end, we begin by discussing why some intuitive contracts may fail.

\subsection{Some Intuitive Contracts}
For simplicity, in this section, we restrict our attention to the scenario when only one customer is present.  For notational ease, when $N=1$, we drop the subscript $i$ referring to the $i$-th customer. 
\begin{example}
Consider a contract that provides a constant payment $c$ to the customer for decreasing her load. Then, the utility function of the customer  is given by:
\begin{displaymath}
V = cu(R)-\beta \frac{(R-x)^2}{2}-h(a),
\end{displaymath}
where $u(.)$ is the unit step function. In this case, the customer seeking to maximize her utility, will choose $a=0$ (i.e., no action) but $R=0^{+}$ (i.e., minimal load reduction reported irrespective of true value of $x$), independently of the value of $c$. The  utility function of the DRA  
is given by
\begin{displaymath}
\Pi=x-cu(R).
\end{displaymath}
Thus, if zero action leads to zero true load reduction, the DRA ends up making a payment in spite of not achieving any load reduction. Thus, this contract is unsuitable for the DRA.
\end{example}
The contract proposed in Example $1$ fails because it does not account for the fact that the amount of effort is known only to the customer and not the DRA. Since the effort is costly, this generates the problem of moral hazard~\cite[Chapter~4]{laf}. To induce a positive load reduction in spite of the presence of moral hazard, the contract must make at least part of the payment proportional to the amount of the load reduction. Otherwise, as discussed above, a rational customer will not choose any non-zero effort.

\begin{example}
Consider a contract in which the DRA provides an incentive $cR$ to the customer in response to the reported reduction $R$ at time $t_{4}$. Then, the utility function of the customer  is given by:
\begin{displaymath}
V = cR-\beta \frac{(R-x)^2}{2}-h(a),
\end{displaymath}
while the utility function of the DRA  
is given by
\begin{displaymath}
\Pi=x-cR.
\end{displaymath}
Especially if  $\beta$ is small, this contract would result in the customer choosing $a=0$ and misreporting a large $R$ to maximize $\E[V]$. Once again, the contract will be unsuitable for the DRA. 
\end{example}
The reason the contract in Example $2$ fails is that the DRA does not have access to the true load reduction $x$ at $t_{4}$ when it has to make at least part of the payment. This creates the problem of {\em adverse selection}~\cite[Chapter~3]{laf}. If the DRA relies on the reported value $R$ for the payment, this creates an incentive for the customer to misreport $R$ as high as possible to gain maximal payment (modulo the falsification cost). 
\begin{remark}
\label{remark22}
If the problem is one that displays only one of  moral hazard or  adverse selection, optimal contracts can be designed using standard methods from the literature. 
However, such contracts are unsuitable  for the problem $\mathcal{P}_{1}$ since we face the problem of moral hazard followed by adverse selection.
\end{remark}

 We conclude this discussion with the following result.


\begin{theorem}
\label{lem0}
Assume that the DRA has accurate knowledge of the true load reduction $x_i$ at time $t_4$.   
 \begin{itemize}
 \item The level of effort $a_i$ by the $i$-th customer which maximizes the utility of the DRA is given by
 \[
 a_i^{\star}=\argmax\E[x_i-h(a_i)].
\]
\item  The DRA can ensure that the effort $a_i^{\star}$ is expended by the each customer $i$ by offering a contract that specifies payments of the form
\begin{equation}
P_i=\E[ x_i-  a_i^{\star}+h(a_{i}^{\star})].
\label{puremhz}
\end{equation}
 \end{itemize}
\end{theorem}

\begin{proof}
See Appendix.
\end{proof}

%

Next, we propose a contract structure for the problem $\mathcal{P}_{1}$ using a two-part payment structure.

\subsection{Proposed Contract Structure}
The contract that the DRA offers to the customers should at once incentivize them to put in costly effort and to report the  load reduction truthfully.  We propose a contract in which the payment to the $i$-th customer is given  by a pair of the form $\{B_{i}(R_i), \alpha_i\}$, where 
\begin{itemize}
\item $B_{i}(R_i)$ is a bonus which is rewarded to the $i$-th customer at $t_{4}$ after the customer reports  load reduction $R_i$, and 
\item $\alpha_i$ is the share of its own gross profit that the DRA realizes due to the demand reduction by the customers and pays back to the $i$-th customer at a much later time $t_{5}\gg t_{4}$. 
\end{itemize}
Note that the payment of the share supposes that the DRA knows the profit it obtains as a result of the load reductions by the customers at time $t_5$.
We first consider the case when this profit is known to the DRA perfectly. We then extend the results to the case  when the gross profit can only be estimated (possibly with some error) in Section \ref{ex2}.

The proposed  contract results in the payment function for the $i$-th customer as given by
\begin{equation}
 P_i = \alpha_i x_i+B_{i}(R_i).
 \label{c1}
\end{equation}
Further, the utility function of the customer  can be written as 
\begin{equation}
V_i=\alpha_i x_i+B_{i}(R_i)-h(a_i)-\beta_{i}\frac{(R_i-x_i)^2}{2},
\label{v1}
\end{equation}
 while the  utility  function for  the DRA is given by
 \begin{equation}
 \Pi =  \sum _{i=1}^{N}(x_i-P_i)=\sum _{i=1}^{N}(1-\alpha_i)x_i- \sum _{i=1}^{N}B_{i}(R_i).
 \label{v2}
 \end{equation} 
By invoking the revelation principle~\cite{laf}, without loss of optimality, we restrict attention to direct mechanisms (where $\hat{x_i}=x_i$) that are incentive compatible. Further, to emphasize the dependence of the utilities on the bonus function and the share, we will sometimes write $V_{i}$ as $V_{i}(B_{i}(R_{i}),\alpha_{i})$ and $\Pi$ as $\Pi(\{B_{i}(R_{i})\},\{\alpha_{i}\}).$

Finally, to ensure that the problem $\mathcal{P}_{1}$ is non-trivial with this contract, we will impose the following further constraints on the problem.
\begin{asum}
\label{asum1}
The DRA does not provide all the profit back to the customers, i.e., $\sum _{i=1}^{N} \alpha_i < 1$. 
\end{asum}
\begin{asum}
\label{asum2}
The bonus is always positive, i.e., $B_{i}(R_i)\geq0$. $B_{i}(R_i)<0$ will imply that the DRA can fine the customers which we disallow in keeping with the individual rationality constraints. We will also assume that $B_{i}(R_i)$ is twice differentiable  and concave in $R_i$  and  further that $B_i(R_i)$  is designed such that $\E_{\mathcal{E}}\left[V_i(B(R^*_i), \alpha_i)\right]$ is concave in $a_i$, where $R^*_i$ is the optimal report by the $i$-th customer as a function of her load reduction.
\end{asum}

\section{Design of the Contract}
\label{sec3}
In this section, we solve for the optimal contract by solving the problem $\mathcal{P}_{1}$. We begin by exploring the design space in terms of identifying the properties that any contract should satisfy. 

\subsection{An Impossibility Result}
The first question that arises is if we can design any contract that incentivize the consumers not to misreport and set  $R_i=x_i$ for all $i$. While naive applications of the revelation principle may suggest that such contracts are not only possible, but that limiting our consideration to such contracts is without loss of generality, this is not the case if the revelation principle is interpreted properly in our context.  Note that the if  a contract that ensures $R_i=x_i$ were possible, Lemma  \ref{lem0} states that the optimal efforts as desired by the DRA are given by  $\{a^{\star}_i\}_{i=1}^{N}$.




\begin{thm}
\label{thfirst}
Under Assumptions~\ref{asum1} and \ref{asum2}, there exists no contract for problem $\mathcal{P}_{1}$ which simultaneously   guarantees  elicitation of the truth from the  customers (in the sense that $R_i=x_i$)  and the realization of the efforts  $\{a_i^{\star}\}$.\end{thm}

\begin{proof}
See Appendix.
\end{proof}

\subsection{ Contract Design}

We now design the payment schemes for the contract  to solve problem $\mathcal{P}_{1}$. 
We solve problem $\mathcal{P}_{1}$ in three steps:
\begin{enumerate}
\item First, we characterize what the optimal value of the load reduction  claimed by each customer would be for a given contract $(\alpha_i, B_{i}(R_i))$. Thus,  we find the optimal value $R_i^*$  of $R_i$, as the solution of the problem
\begin{align}
R_i^*
&= \argmax_{R_{i}} \E_{\mathcal{E}_{-i}}[V(B_{i}(R_i), \alpha_i )].
\label{Rstar}
\end{align}
\item Then, for this value  $R_i^*$,  we calculate the optimal effort $a_i^*$ exerted by the customers, i.e.,  we solve the problem 
\begin{equation}
a_i^*= \argmax_{a_{i}} \E_{\mathcal{E}}[V(B_{i}(R_i^*), \alpha_i )].
\label{effstar}
\end{equation}
\item Finally, having characterized the response of the customers, we optimize the parameters of the proposed contract for the DRA. Thus, we solve  \begin{equation}
 \{B^*(.),\alpha^*_i\} = \argmax_{\{B_{i}(.)\},\{\alpha_{i}\}} \E_{\mathcal{E}}[\Pi(\{B_{i}(R_i^*), \{\alpha_{i}\})],
 \label{optpay}\end{equation}
\end{enumerate}
when the customers exert the efforts $\{a_i^{*}\}$ and report reductions $\{R_i^{*}\}$. We continue with the following result on the first step.
\begin{thm}
\label{pr0}
Consider the optimization problem $\mathcal{P}_1$. The optimal choice of the reported load reduction $R_i$ obtained as a solution to the problem  \eqref{Rstar} is given by the solution to the following equation
\begin{equation}
R^*_i-x_i=\frac{1}{\beta_{i}}\frac{\partial \E_{\mathcal{E}_{-i}}[B_{i}(R_i)]}{ \partial R_i}\bigg\rvert_{R_{i}=R^*_{i}}.
\label{Rstar2}
\end{equation}
\label{pro3}
\end{thm}
\begin{proof}
See Appendix.
\end{proof}

 This result characterizes the optimal reporting by the customer. We note the following interesting feature.
\begin{cor}
With the payment scheme $P_i{=} \alpha_i x_i+B_{i}(R_i)$  in problem $\mathcal{P}_{1}$, if  $B_{i}(R_i)$ is a decreasing (respectively increasing) in $R_i$, then the customer underreports (respectively overreports) her true load reduction.\label{corr1}
\end{cor}
 \begin{proof}
 The proof follows directly  from \eqref{Rstar2}. 
 \end{proof}
 
 \begin{remark}
Since $B_{i}(R_{i})$ may be a decreasing and an increasing function for different values of $R_i$, the optimal contract may induce both \textit{under-reporting} and  \textit{over-reporting} of the load reduction by the customer. In other words, for some values of the true load reduction, it is possible that  a strategic customer may decrease her base load before the DR event and under-report her power reduction to maximize her profit.  
 \end{remark}

Next, we characterize the optimal effort by solving the problem \eqref{effstar}.
\begin{thm}
\label{pro22}
Consider the optimization problem $\mathcal{P}_1$. The  optimal choice of the effort $a_i$ is obtained as the solution of  the equation
\begin{equation}
a_i^*=\alpha_i+\frac{\partial \E_{\mathcal{E}}\left[B_{i}(R^*_i)-\frac{1}{2\beta_i}\left(\frac{\partial B_{i}(R_i)}{\partial R_i}\big\rvert_{R_{i}=R_{i}^{*}}\right)^2\right]}{\partial a_i}\Bigg\rvert_{a_{i}=a_{i}^{*}},
\label{effortstar}
\end{equation}
where $R_{i}^{*}$ is as specified in Theorem~\ref{pr0}. 
\end{thm}
\begin{proof}
See Appendix.
\end{proof}

 Finally, having characterized the response of the customers, the third step is to optimize the parameters of the proposed contract by solving  \eqref{optpay}. 
 \begin{thm}
 \label{proposition_opt_contract}
Consider the problem formulation in Section \ref{sec1} and the optimization problem $\mathcal{P}_1$. The  optimal choice of the share  assigned to $i$-th customer and the optimal bonus function are defined implicitly through the equations 
\begin{align}
\label{optalfa}
\alpha_i^*&=1-\frac{a_i^*+\frac{\partial \E_{\mathcal{E}}[B_{i}(R_i)]}{\partial \alpha_i}\Big\rvert_{\alpha_{i}=\alpha_{i}^{*}}}{\frac{\partial a^*_i}{\partial \alpha_i}\Big\rvert_{\alpha_{i}=\alpha_{i}^{*}} }\\
 B^*(R_i^{*})&= \argmax\left[ (1-\alpha^*_i)a_i^*-\E_{\mathcal{E}}[B^*(R_i^*)]\right], 
 \label{bp3}
 \end{align}
  where $a_i^*$, and $R_i^*$ are evaluated using \eqref{Rstar2} and \eqref{effortstar}.
  \end{thm} 
 \begin{proof}
Proof follows directly from \eqref{optpay}.
 \end{proof}

 \subsection{Example Contracts}
Notice that equations \eqref{Rstar2}-\eqref{bp3}  do not  constrain the choices of the contract terms or the resulting actions of the customers to be unique. We now make more assumptions on the problem and provide some example contracts that result.  We consider two scenarios:
\begin{itemize}
\item {\em Unspecified load reduction:} In the first scenario, we consider  the case when the DRA is interested in the overall load  reduction from all the customers to be as large as possible. In this case, we propose a bonus function of the form $B_i(R_i)=\mu (R_i -c)$ for every  customer $i$ with , where $c\geq0$ is a specified constant.
\item {\em Specified load reduction:} In the second case, we assume that the DRA wishes the overall load reduction to be equal to a given value $\Gamma$. In this case, the customers are in competition with each other for the load reduction they provide and the consequent payment they obtain. Thus, we must consider the bonus function to customer $i$  to be  a  function of not only her own report  $R_i$, but also the reports from other customers. Following the classical Cournot game \cite{cournot}, we propose a bonus function of the form  $B_i(\{R_{i}\})=R_i(\lambda- \sum_{j=1}^{N}R_j)$, where $\lambda$ is a designer-specified parameter that depends on $\Gamma$.
\end{itemize}


\subsubsection{Unspecified load reduction}
We begin with the case when the bonus function is of the form $B_i(R_i)=\mu (R_i -R_0)$ for all customers. In this case, there is no competition among the customers. Our first result says that we can simplify the incentive compatibility constraint.
 \begin {lem}
 \label{pro1}
 Consider the problem $\mathcal{P}_{1}$ such that $\forall i$, the bonus function does not depend on $R_{j}, j\neq i$ and is further of the form $B_i(R_i)=B(R_i)$. If the proposed contract structure in \eqref{c1} is incentive compatible, then  it holds that $\forall i,$
\begin{equation*}
  \frac{\partial B(R_i)}{\partial x_i}=\beta_i (R_i-x_i)\frac{\partial R_i}{\partial x_i} \qquad \textrm{ and }\qquad
\frac{\partial R_i}{\partial x_i}\geq0.
\end{equation*}
 In particular for the contract  $B(R_i)=\mu (R_i-R_0)$, these conditions reduce to $$\frac{\partial R_i}{\partial x_i}(\mu-\beta_i (R_i-x_i))=0\qquad \textrm{ and }\qquad\frac{\partial R_i}{\partial x_i}\geq0.$$
 \end{lem}
 \begin{proof}
See Appendix.
\end{proof}
 \begin{remark}
 The result implies that an incentive compatible contract will associate higher load reduction $x_i$ with a higher report $R_i$.
 \end{remark}
 
 With this result, we can restate the problem to be solved by the DRA as
  \begin{equation*}
\textrm{$\mathcal{P}_{2}$:}
\begin{cases}
&\underset{{\{\mu,\{\alpha_{i}\}\}}}\max  \E_\mathcal E[\Pi]\\
s.t. &\textrm{$\{a_{i}, R_{i}\}$ is chosen to maximize $\E_\mathcal E[V_i]$ by each}\\ &\textrm{ customer $i$}\\
& \mu \geq 0\\
& \textrm{Individual rationality constraint: } \E[V_i]\geq0\\
& \textrm{Incentive compatibility constraints: } \mu=\beta_{i} (R_i-x_i)\\
&\qquad\qquad\qquad\qquad\qquad\qquad\qquad \frac{\partial R_i}{\partial x_i}\geq0.
\end{cases}
\end{equation*}

The following result summarizes the optimal contract and the resulting actions under it for the problem $\mathcal{P}_{2}$.
\begin{thm}
\label{pro44}
Consider the problem $\mathcal{P}_2$ posed above. 
\begin{itemize}
\item The optimal contract obtained as a solution to the problem is specified by the relations
\begin{align*}
\mu^*&=\frac{R_0 \beta_i}{2}\\
\alpha^*_i&=0.5-\mu^*.
\end{align*}
\item In response to this optimal contract, every customer $i$ over-reports her true load reduction as $R^*_i=x_i+\frac{\mu^*}{\beta_i }$. Further, the customer exerts the effort $a^*_i=\mu^*+\alpha_i^*.$
\end{itemize}
\end{thm}
\begin{proof}
See Appendix.
\end{proof}
\begin{remark}
Note that the constraint that $\alpha^*_{i}\geq 0$ implies the condition $c\beta_{i}\leq 1.$
\end{remark}



\subsubsection{Specified load reduction}
 We now consider the case when $B_i(R_i, \sum _{j=1}^{N}R_j)=R_i(\lambda-\sum _{j=1}^{N}R_j)$. Once again, we can simplify the incentive compatibility constraint according to the following result. 
\begin{lem}
\label{pro551}
Consider the problem $\mathcal {P}_1$ with a bonus function for the $i$-th customer that depends on the report $R_{i}$ submitted by the $i$-th customer and the sum of the reports $\sum_{j=1}^{N} R_j$ submitted by all other customers. If the proposed contract is incentive compatible, then it holds that  
\begin{align*}
&\frac {\partial {\E_{\mathcal{E}_{-i}}\left[B(R_i, \sum_{j=1}^{N} R_j)\right]}}{\partial {x_i}}= \beta_i(R_i-x_i)\frac {\partial {R_i}}{\partial {x_i}}\\
\nonumber &\frac{\partial R_i}{\partial x_i}\geq0.
\end{align*}In particular for the contract  $B_i(R_i, \sum _{j=1}^{N}R_j)=R_i(\lambda-\sum _{j=1}^{N}R_j)$, these conditions reduce to\begin{align}
&\frac{\partial R_i}{\partial x_i}\left[\lambda+\beta_i x_i-(\beta_i+2)R_i-\sum\limits_{\substack{{j=1}\\{j\neq i}}}^{N} \E_{\mathcal{E}_{-i}}[R_j]\right]= R_i \frac{\partial \sum\limits_{\substack{{j=1}\\{j\neq i}}}^{N} \E_{\mathcal{E}_{-i}}[R_j]}{\partial x_i} \label{IC2}, \\\nonumber& \frac{\partial R_i}{\partial x_i}\geq0.
\end{align}
\end{lem}
\begin{proof}
See Appendix.
\end{proof}





 
We can now restate the problem $\mathcal{P}_{1}$ to be solved by the DRA as follows.
\begin{equation*}
\textrm{$\mathcal{P}_{3}$:}
\begin{cases}
&\underset{\{\alpha_{i}\}}\max  \E_\mathcal E[\Pi]\\
s.t. &\textrm{$\{a_{i}, R_{i}\}$ is chosen to maximize $\E_\mathcal E[V_i]$ by each}\\ &\textrm{ customer $i$}\\
& \textrm{Individual rationality constraint: } \E[V_i]\geq0\\
& \textrm{Incentive compatibility constraints specified by \eqref{IC2}} \\
&\textrm{expected overall load reduction = }\Gamma.
\end{cases}
\end{equation*}

Since the bonus paid to $i$-th customer is a function  not only of $R_i$, but also of the reports from the other customers, the customers compete against each other to gain the maximum compensation possible. Thus, the optimal strategies of the players become interdependent. We analyze this interdependence in the usual Nash Equilibrium sense. For the following result, we make the simplification that all the parameters $\alpha_i$'s and $\beta_i$'s are constants with $\alpha_i=\frac{\alpha}{N}$ and $\beta_i=\beta$, $\forall i$. 
 

We first present the following initial result.

\begin{thm}
\label{theorem2}
Consider the problem $\mathcal P_{3}$. Define the variables
\begin{align*}
A&=\frac{\beta+F}{\beta+1+N}\\
B&=\frac{\beta(\beta+F)}{(\beta+1)(\beta+1+N)}\\
C&=1+\frac{2(\beta+F)^2}{(\beta+2)^2}-\frac{2(\beta+F)F}{(\beta+2)}+\frac{\beta(4-4F+F^2)}{(\beta+2)^2}\\
D&=\frac{(\frac{\beta}{\beta+1+N})^2}{C+(N-1)B}\\
E&=\frac{N}{\beta+1+N}\\
F&= \frac{\beta (N-1)}{(\beta+1+N)(\beta+1)}.
\end{align*}
There is a unique Nash Equilibrium among the users and the DRA as given by the following: 
\begin{enumerate}[label=(\roman*)]
\item The DRA selects the contract as
\begin{align}
\label{eq:share_shared} \alpha_{i}^*&=\frac{1-\lambda[A(1-2ND)+\frac{\beta}{\beta+1+N}(1-2E)]}{2(1-ND)}\\
\lambda^*&=\frac{(C+(N-1)B)\Gamma-N\alpha_{i}^*}{AN}.
\label{eq:bonus_shared}
\end{align}
\item The customers exert the optimal effort and report as
\begin{align}
\label{eq:action_shared}
a_i^*&= \frac{\alpha_i+A\lambda-B\sum\limits_{\substack{{j=1}\\{j\neq i}}}^{N} a_j }{C}=\frac{\alpha_i+A\lambda}{C+(N-1)B}\\
R^*_i&=\frac{\lambda-\left(G-F a_i^*\right)+\beta x_i}{\beta+2},
\label{eq:report_shared}
\end{align}
where
$$
G=\frac{\lambda(N-1)}{\beta+1+N}+ \frac{\beta(\beta+2)\sum\limits_{\substack{{j=1}\\{j\neq i}}}^{N} a_j^*}{(\beta+1)(\beta+1+N)}.$$
\end{enumerate}
This equilibrium always exists.
\end{thm}
\begin{proof}
See Appendix.
\end{proof}
\begin{remark}
To obtain the conditions for $\alpha_{i}^*\geq0$, we can substitute \eqref{eq:share_shared} in \eqref{eq:bonus_shared} to  obtain
\begin{align*}
\Gamma&=N\frac{\alpha_i^*+A\frac{1-\alpha^*2(1-ND)}{A(1-2ND)+\frac{\beta (1-2E)}{\beta+1+N}}}{C+(N-1)B}\\
\Rightarrow\alpha_i^*&=\frac{A-\Gamma \frac{(C+(N-1)B)[A(1-2ND)+\frac{\beta(1-2E)}{\beta+1+N}]}{N}}{\frac{\beta (2E-1)}{\beta+1+N}+A}.
\end{align*}
Note that the denominator evaluates to 
\begin{align*}
&\frac{\beta (2E-1)}{\beta+1+N}+A\\
&=\frac{\beta}{\beta+1+N} \frac{N-\beta-1}{\beta+N+1}+\frac{\beta+F}{\beta+1+N}\\
&=\beta\left(\frac{2N}{(\beta+N+1)^2}+\frac{(N-1)}{(\beta+1+N)^2(\beta+1)} \right)\\& \geq 0.
\end{align*}
Thus, the condition for $\alpha^*_i\geq0$ is refined to the choice of $\Gamma$, $N$ and $\beta$ which satisfy 
\[
\Gamma\leq \frac{NA} {(C+(N-1)B)[A(1-2ND)+\frac{\beta(1-2E)}{\beta+1+N}]}.
\]
The condition implies that as the desired load reduction $\Gamma$ increases, the number of customers $N$ that the DRA contracts with must increase as well.
\end{remark}
%
%
%
%
%
%
\begin{remark}
\label{rem.int}
Note that the optimal level of  the effort $a_i^*$ expended by the $i$-th customer is an increasing function of both the assigned  share $\alpha$ and the total amount of desired load reduction $\lambda$, which is intuitively specifying. \end{remark}

 \subsection{Extensions} 
\label{sec5}
Although the above development was done with some specific assumptions, the contracts can be generalized to remove many of these assumptions. We provide some examples below. For notational ease, we consider the case when $N=1$ and drop the subscript $i$ referring to the $i$-th customer. Further, we assume that the parameter $\beta=1$ and the bonus function is given by $B(R)=R(\lambda-R).$ 

\subsubsection {Realization error with non-zero mean}
\label{ex1}
The effort $a$ by the customer is assumed to lead to the realization of load reduction $x$. As specified by Assumption \ref{asum3}, in the development so far, we assumed that the realization error $e=x-a$ is a random variable with mean zero. If, instead, the error has mean $m_{e}$, then the following result summarizes the optimal contract. 
\begin{pro}
Consider the problem $\mathcal{P}_{3}$ for $N=\beta=1$ and mean $m_{e}$ of the realization error. 
\begin{itemize}
\item The optimal contract is given by 
\begin{align*}
\alpha^*&=\frac{7.5-3 \lambda+4.3m_e}{14}\\
\lambda^*&=5\Gamma-3\left(\alpha^*+m_{e}\right).
\end{align*}
\item The optimal effort and the report by the customer are given by
\begin{align*}
a^*&=\frac{3\alpha+ \lambda-2m_e}{5}\\
R^*&=\frac{\lambda+ x}{3}.
\end{align*}
\end{itemize}
\end{pro}
\begin{proof}
The proof follows in a straight-forward manner along the lines of that of Theorem~\ref{theorem2}. 
\end{proof}
\begin{remark}
Note that the optimal reporting function $R^*$ does not depend on $m_{e}$.  Further, as $m_e$ increases (resp. decreases), 
\begin{itemize}
\item the expected load saving for the same contract increases (resp. decreases),
\item the optimal effort exerted by customer is lower (resp. higher),
\item the optimal value $\alpha^*$ of the share provided by DRA to the customer will increase (resp. decrease).
\end{itemize}
\end{remark}



\subsubsection {Inexact knowledge of  the true load reduction}
\label{ex2}
So far, we assumed that at $t_{5}$, the DRA has an accurate knowledge of the true load reduction $x$ due to the customer. In practice, it may only be able to estimate this reduction by, e.g., large scale data analysis on all similar customers on that day or historical behavioral of the same customer. Let the DRA observe a noisy estimate $y=x+n$ of the load reduction at $t_{5}$, where $n$ denotes the estimation error. We assume that this error in independent of $X$ and has mean $m_n$.  In this case, the share of the profit assigned to the customer changes to $\alpha y$. In other words, the utility functions of the customer and the DRA from \eqref{v1} and \eqref{v2} alter to

\begin{align}
\label{v1_new}V&=\alpha y+B(R)-h(a)-\beta\frac{(R-x)^{2}}{2}\\
\label{v2_new}\Pi&=(1-\alpha) y-B(R).
\end{align}
We have the following result that can be proved along the lines of Theorem~\ref{theorem2}.
\begin{pro}
Consider the problem $\mathcal{P}_{3}$ for $N=\beta=1$ and with $n$ denoting the error in estimating the load reduction $x$ at $t_{5}$, so that the utility functions of the customer and the DRA are given by~(\ref{v1_new}) and~(\ref{v2_new}). 
\begin{itemize}
\item The optimal contract is given by 
\begin{align*}
\alpha^*&=\frac{7.5-3 \lambda-12.5m_n}{14}\\
\lambda^*&=5\Gamma-3\alpha^*.
\end{align*}
\item The optimal effort and the report by the customer are given by
\begin{align*}
a^*&=\frac{3\alpha+ \lambda}{5}\\
R^*&=\frac{\lambda+ x}{3}.
\end{align*}
\end{itemize}
\end{pro}
\begin{remark}
Note that the optimal reporting function $R^*$ and the optimal effort $a^*$ do not depend on $m_{n}$.  Further, as $m_n$ increases (resp. decreases),  the optimal value $\alpha^*$ of the share provided by DRA to the customer decreases (resp. increases). 
\end{remark}

%


\section{ illustration and discussion}
\label{illus}

\begin{figure}[!htb]
\centering
\includegraphics[width=7cm, height=4cm]{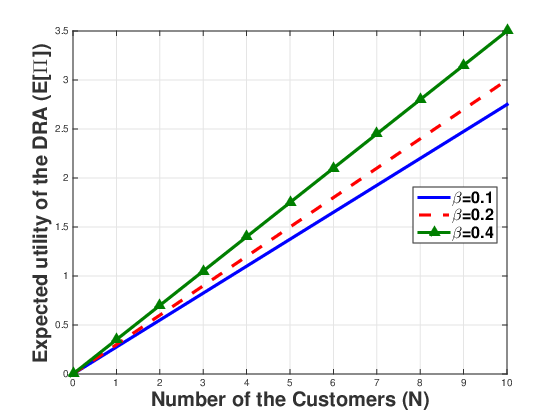}
\caption{Expected utility of the DRA for various values of $\beta$ as a function of number of the customers $N$ for the specified load reduction scenario. }
\label{fig:digraph}
\end{figure}

\begin{figure}[h!]  
\begin{subfigure}{.5\textwidth}  
  \centering  
  \includegraphics[width=7cm, height=4cm] {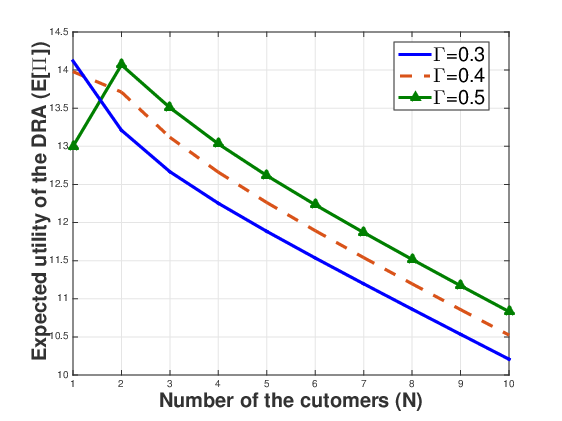}
    \caption{}   
    \label{uDRA(N)}
\end{subfigure}  

\begin{subfigure}{.5\textwidth}  
  \centering  
  \includegraphics[width=7cm, height=4cm]  {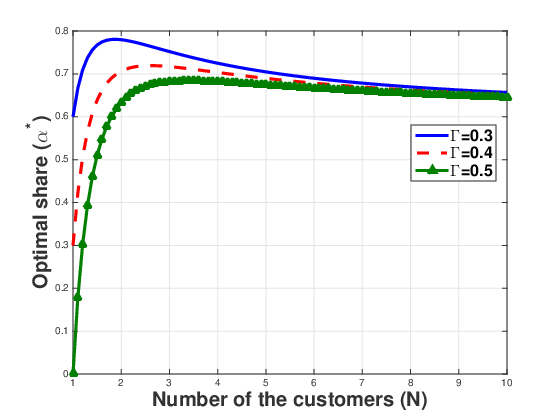}
    \caption{}   
       \label{ALPHA-N}
\end{subfigure}   

\begin{subfigure}{.5\textwidth}  
\centering
  \includegraphics[width=7cm, height=4cm]{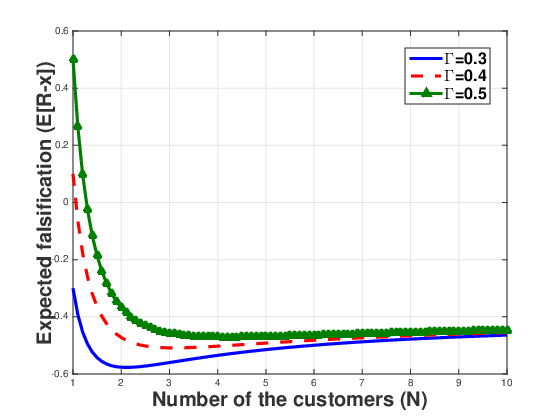} 
    \caption{}   
    \label{fals-N} 
\end{subfigure}  

\caption{Expected utility of the DRA (a), optimal value of the share to the customers (b), and expected value of the falsification by the customer as a function of $N$ for various $\Gamma$, under the bonus function $B_i=R_i(\lambda-\sum _{j=1}^{N}R_j)$,  given $\beta=1$.}  
\end{figure}

We now present some illustrative numerical examples. We first consider the unspecified load reduction scenario. We set $R_{0}=1$ and assume that $\beta_{i}=\beta$, $\forall i$. Figure \ref{fig:digraph} presents the expected utility of the DRA with the optimal contract as presented in Theorem~\ref{pro44} for various values of $\beta$ as we vary the number of the customers $N$ that the DRA contracts with.  As shown in this figure, for a fixed value of $\beta$, the expected utility of the DRA  increases linearly with the number of customers. This is intuitively specifying since as specified by Theorem~\ref{pro44}, for a fixed $\beta$, the effort invested by each customer for the optimal contract is a constant independent of $N$ or $\mu$.  Further, we observe that for a fixed $N$, as $\beta$ increases, the expected utility of the DRA increases. In other words, for the same expected load reduction, the DRA needs to pay less to the customers. Note that satisfying the condition for $\alpha^* \geq 0$ precludes the choice of an arbitrarily large $\beta$ by the DRA.

Next, we consider the specified load reduction scenario. We set $\beta=1$.  Figure  \ref{uDRA(N)} shows how the expected utility of the DRA varies as a function of the number of customers contracted by the DRA for various value of the total expected load reduction $\Gamma$ that the DRA desires. We can observe that for a small number of customers,  the expected utility of the DRA  is a decreasing function of $\Gamma$ while for large enough $N$, it is an increasing function of $\Gamma$. Intuitively, if the number of customers that the DRA has contracted with is too small, it must pay too high a compensation for realizing the desired load reduction. In fact, as the expected load reduction that it wishes increases, the expected utility of the DRA may become negative unless the number of customers is also increased. Once a sufficient number of customers have been contracted with, the total payment once again decreases with the number of customers. Once again, this does not imply that the number of customers can be increased arbitrarily given the constraints of $\alpha^* \geq 0$ and individual rationality for the customers. Viewed alternatively, for the same number of customers, the total expected load reduction $\Gamma$ is bounded by these constraints.

For the same setting, the optimal value of the share $\alpha^*$ assigned to the customer as a function of the number of customers for various value of $\Gamma$ is  illustrated in Figure  \ref{ALPHA-N}.  The plot indicates that the optimal value of the share assigned to the customer by the DRA is always positive as desired. Further, it is a decreasing function of  the expected load reduction $\Gamma$ desired by the DRA. This plot illustrates that the constraint $\alpha^*>0$ imposes an upper bound on the accepted value of $\Gamma$. For a given value of $\Gamma$, the variation of the optimal share is non-intuitive, although it should be noted that for a large enough number of customers, the share converges to the same value.

Figure \ref{fals-N}  displays the expected value of the falsification by each customer as a function of the number of customers for various value of $\Gamma$.   Figure \ref{fals-N}  implies that the expected value of the falsification decreases when the DRA chooses a larger $\Gamma$; in fact, it becomes negative for a high enough $\Gamma$. The negative falsification is interesting since it implies that the customer {\em under-reports} her load reduction. Note that a larger value of $\Gamma$ can be interpreted as a higher expected utility of the DRA. Thus, as the expected utility of the DRA increases, the customers 
under-report their true load reduction since they can gain more compensation through their shares. 


Finally, we illustrate the impact of the realization error $m_e$  and the estimation  error $m_n$. We consider a single customer and set $\beta=1$. Figure     \ref{payment-me} plots the expected payment by the DRA as a function of $m_{e}$. As can be seen, the pattern of variation is quite complex. For a large enough $m_{e}$, the expected load reduction by the customer is large. Thus, the payment through the shares dominates and the expected payment also increases. 
Figure     \ref{payment-mn} plots the expected payment by the DRA as a function of $m_{n}$. The figure illustrates that as the mean of the error with which the DRA estimates the true load reduction increases, it increases the compensation it provides to the customer. In addition, as the DRA wishes to realize a larger value of $\Gamma$, the expected value of the compensation also becomes larger. This is expected since a higher value of $m_{n}$ implies that the DRA observes a higher load reduction compared to the one realized in practice; consequently, it rewards the customer more based on its own observation.

\begin{figure}[h!]  
\begin{subfigure}{.5\textwidth}  
\centering
  \includegraphics[width=7cm, height=4cm]{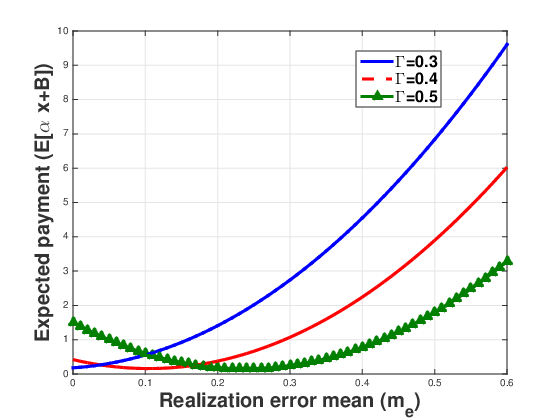}
    \caption{}  
    \label{payment-me} 
\end{subfigure}  

\begin{subfigure}{.5\textwidth}  
\centering
  \includegraphics[width=7cm, height=4cm]{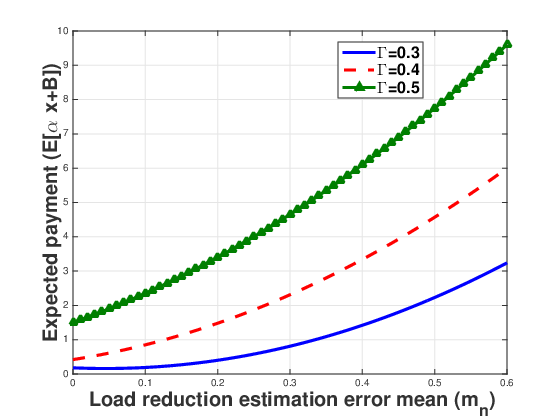}
    \caption{}  
        \label{payment-mn}  
\end{subfigure} 
\caption{The expected payment to the customer  $(a)$ as  a function of $m_e$ and $(b)$ as a function of $m_n$, for different values of $\Gamma$, given $N=1$. }
\end{figure}


\section{Conclusion and future directions}
\label{concl}
In this paper, we designed an optimal contract between a demand response aggregator (DRA) and power grid customers for incentive-based demand response. We considered a setting in which the DRA asks the customers to reduce their electricity consumption and compensates them for this demand curtailment. However, given that the DRA must supply every customer with as much power as she desires, a strategic customer can temporarily increase her base load to report a larger reduction as  part of the demand response event. The DRA wishes to incentivize the customers  both to  make  costly effort to reduce load and to not falsify the reported load reduction. We modeled this problem  as a contract design problem and presented a solution. The proposed contract consists of a part that depends on (the possibly inflated) load reduction as measured by the DRA and another that provides a share of the  profit that accrues to the DRA through the demand response event to the customers. The contract design, its properties, and the interactions of the customers under the contract were discussed and illustrated.

The paper opens many directions for future work. We can consider the dynamic case when the customers need to be incentivized to participate in demand response repeatedly. Particularly interesting is the case when the customers also gain a signal about the total load reductions on a particular day and can alter their strategies accordingly. Another problem that should be considered  is when the DRA is able to observe only the sum of the profit due to the effort by multiple customers and thus the payment cannot be based on individual efforts of each customer. This may lead to the problem of `free-riding' in which some customers seek payment in spite of not putting in any effort, by relying on the efforts of other customers.  

\appendices
\section{}
\label{appendix_1}
\textbf{Proof of Lemma \ref{lem0}}
\begin{proof}
If the DRA has accurate knowledge of the true load reduction $x_i$ at time $t_4,$ the utility of the DRA is  given by $\Pi=\sum_{i=1}^{N}(x_i-P_i)$ and that of the $i$-th customer is given by $V_i=P_i-h(a_i)$. Thus, for any effort $a_{i}$ exerted by the customer, the payment $P_i=h(a_i)$ would solve the problem $\mathcal{P}_{1}$ (notice, in particular, that the individual rationality constraints will be satisfied with this payment). Substituting this payment in the utility of the DRA, we see that the 
%
level of effort $a_i$ by the $i$-th customer which maximizes the expected utility of the DRA is given by
 \[
 a_i^{\star}=\argmax \E[x_i-h(a_i)].
\]
Further, with this effort, the expected utility of the DRA is given by
\[
\E[\Pi^{\star}]=\sum_{i=1}^{N}\left(a_{i}^{\star}-h(a_{i}^{\star})\right),
\]
where we have used Assumption~\ref{asum3}. However, this payment can be implemented only if the DRA could observe $a_{i}$. We show that even if $a_i$ is unobservable for the DRA, and it can only observe $x_i$ at time $t_4$, it can incentivize the customer to exert the same effort and realize the  maximal utility $\E[\Pi^{\star}]$ for itself. To this end, consider the payment specified by
\begin{equation*}
P_i=x_i-  a_i^{\star}+h(a_{i}^{\star}).
\end{equation*}
With this payment, the expected utility of the $i$-th customer can be written as 
\[
\E[V_i]=\E[x_i-a_i^{\star}+h(a_{i}^{\star})]-h(a_i)=\E[x_i-h(a_i)]-\left(a_i^{\star}-h(a_{i}^{\star})\right).
\] 
Given the definition of $a_i^{\star}$, it is easy to see that the customer $i$ chooses $a_i=a_i^{\star}$ to maximize her expected utility. Further, the expected utility of the DRA is given by 
\[
\Pi=\sum_{i=1}^{N}\E[x_{i}-P_{i}]=\sum_{i=1}^{N}\E[a_i^{\star}-h(a_{i}^{\star})]=\E[\Pi^{\star}]
\]
Thus, the expected utility of the DRA is maximized with this choice of the payment.
\end{proof}

\textbf{Proof of Theorem \ref{thfirst} }
\begin{proof}
We prove by contradiction. Suppose that there exists a bonus function  $B_{i}(R_i)$  and an allocation $\{\alpha_{i}\}$ which simultaneously satisfies two conditions: (i) $\mathcal{C}_{1}:$ it incentivizes each customer $i$ to choose the strategy $R(x_i)=x_i$, and (ii) $\mathcal{C}_{2}:$ it incentivizes each customer to choose $a_i=a_i^{\star}$.  

By $\mathcal{C}_{1}$, the utility of the $i$-th customer is maximized if she reports  $R_i=x_i$. Since the portion $\alpha_i x_i$ of payment does not depend on $R_i,$ we can write 
\begin{align}
\nonumber&\frac{\partial{\E_{\mathcal E_{-i}}[V_i]}}{\partial{ R_i}}\biggr\rvert \underset{R_i=x_i}\:\:\:=\frac{\partial \E_{\mathcal E_{-i}}[B_{i}(R_i)]}{\partial R_i}-\beta_i(R_i-x_i)\biggr\rvert\underset{R_i=x_i}\:\:\:=0,\\
\label{eq:proof_thm1_cond1}&\qquad\qquad\qquad\Rightarrow \frac{\partial \E_{\mathcal E_{-i}}[B_{i}(R_i)]}{\partial R_i}=0\\
\nonumber&\frac{\partial^2{\E_{\mathcal E_{-i}}[V_i]}}{\partial{R_i}^2}\biggr\rvert  \underset{R_i=x_i}\:\:\:=\frac{\partial ^2\E_{\mathcal E_{-i}}[B_{i}(R_i)]}{\partial R_i^2}-\beta_i\biggr\rvert\underset{R_i=x_i}\:\:\:\leq0\\
\label{eq:proof_thm1_cond2}&\qquad\qquad\qquad\Rightarrow \frac{\partial ^2\E_{\mathcal E_{-i}}[B_{i}(R_i)]}{\partial R_i^2}\leq \beta_i.
\end{align}
(\ref{eq:proof_thm1_cond1}) implies that 
\[\E_{\mathcal E_{-i}}\bigg[\frac{\partial B_{i}(R_i)}{\partial R_i}\bigg]=0,\]
or, in turn, that
\[\int_{\mathcal E_{-i}}\frac{\partial B_{i}(R_i)}{\partial R_i}f_{X_{-i}}dX_{-i}=0.\]
Since this equation should hold for all $R_{i}$, we must have $B_{i}(R_i)=c$ for some constant $c$. In other words, the DRA  provides  a fixed compensation to the customer  irrespective of what she reports. Further, with truthful reporting, Lemma~\ref{lem0} implies that the payment
\begin{equation}
\label{eq:temp_thm_1}
P=\sum_{i=1}^{N}P_{i}=\sum_{i=1}^{N}\left(x_i-  a_i^{\star}+h(a_i^{\star})\right)
\end{equation}  
maximizes the utility of the DRA while ensuring choice of the desired action $a_i^{\star}$  by the customers. The fact that $B_{i}(R_i)=c$ and that the payment is given by~(\ref{eq:temp_thm_1}) implies that the following conditions must be met
\begin{equation*}
B_{i}(R_i)=-\sum_{i=1}^{N} a_i^{\star}\qquad\textrm{ and }\qquad
\sum _{i=1}^{N}\alpha_i=1.
\end{equation*}
However, this allocation violates Assumptions \ref{asum1} and \ref{asum2}. Thus, our supposition is wrong and there does not exist a payment function that simultaneously  guarantees $\mathcal{C}_{1}$ and $\mathcal{C}_{2}.$
\end{proof}

\textbf{Proof of Theorem \ref{pr0}}
\begin{proof}
We can write~(\ref{Rstar}) as
\begin{multline*}
R_i^*=\argmax_{R_{i}} \E_{\mathcal{E}_{-i}}[V_{i}(B_{i}(R_i), \alpha_i)]\\
=\argmax_{R_{i}}[\alpha_i x_i+\E_{\mathcal{E}_{-i}}[B_{i}(R_i)]-\beta_i \frac{(R_i-x_i)^2}{2}- h(a_i)].
\end{multline*}
For optimality, we set 
\begin{align*}
&\frac{\partial \E_{\mathcal{E}_{-i}}[V_{i}(B_{i}(R_i), \alpha_i)]}{\partial R_i}=0\\
 \Rightarrow &\frac{\partial \E_{\mathcal{E}_{-i}}[B_{i}(R_i)]}{ \partial R_i}-\frac{\beta_i}{2} \frac{\partial (R_i-x_i)^2}{ \partial R_i}=0\\
\Rightarrow &R^*_i-x_i=\frac{1}{\beta_{i}}\frac{\partial \E_{\mathcal{E}_{-i}}[B_{i}(R_i)]}{ \partial R_i}\bigg\rvert_{R_{i}=R^*_{i}}.
\end{align*}
Note that the second derivative
\[
\frac{\partial^2 \E_{\mathcal{E}_{-i}}[V_i]}{\partial R_i^2}=\frac{\partial^2 \E_{\mathcal{E}_{-i}}[B_{i}(R_i)]}{\partial R_i^2}-\beta_i<0,
\] 
 given the concavity of $B_{i}(R_i)$. Thus, $R_i^*$ that satisfies \eqref{Rstar2} is indeed a maximizer.
 \end{proof}

\textbf{Proof of Theorem \ref{pro22}}
\begin{proof}
The effort $a_i$ is chosen to maximize the  utility of the customer given that the optimal report is calculated as given in the equation~\eqref{Rstar2}. Thus, 
\begin{align*}
a_i^*&=\argmax_{a_{i}} \E_{\mathcal{E}}\left[V_i(B_{i}(R^*_i), \alpha_i)\right]\\
&=\argmax_{a_{i}} \E_{\mathcal{E}}\left[\alpha_i x_i+B_{i}(R_i^{*})-\beta_i \frac{(R_i^{*}-x_i)^2}{2} - \frac{a_i^2}{2}\right]\\
&=\argmax_{a_{i}} \left[\alpha_i a_i+\E_{\mathcal{E}}\left[B_{i}(R_i^{*})-\frac{\beta_i}{2} (R_i^{*}-x_i)^2\right] - \frac{a_i^2}{2}\right],
\end{align*}
where we have used Assumption~\ref{asum3}.
For optimality, we set 
\[
\frac{\partial \E_{\mathcal{E}}\left[V_i(B_{i}(R^*_i), \alpha_i)\right]}{\partial a_i}=0.\]
This condition yields
 \begin{equation} \alpha_i+\frac{\partial \E_{\mathcal{E}}\left[B_{i}(R_i^{*})-\frac{\beta_i}{2} (R_i^{*}-x_i)^2\right]}{ \partial a_i}-a_i=0.
 \label{astar11}
\end{equation} 
Using Theorem~\ref{pr0}, we can write this condition as
\[a_i^*=\alpha_i+\frac{\partial \E_{\mathcal{E}}\left[B_{i}(R^*_i)-\frac{1}{2\beta_i}\frac{\partial B_{i}(R_i)}{ \partial R_i}\bigg\rvert_{R_{i}=R^*_{i}}\right]}{ \partial a_i}\Bigg\rvert_{a_{i}=a_{i}^{*}}.\]
Finally, given the concavity of $\E_{\mathcal{E}}\left[V_i(B_{i}(R^*_i), \alpha_i)\right]$ in $a_i$, we note that $a_i^*$ is a maximizer.  


\end{proof}

\textbf{Proof of Lemma \ref{pro1}}
\begin{proof}
\eqref{v1} implies that the utility of the $i$-th customer, $V(B(R_i), \alpha_i)$, depends on the parameter $x_{i}$ through the report $R_{i}$ and the bonus $B(R_{i})$. For an incentive compatible contract, the utility of the $i$-th customer  is maximized when she chooses to calculate her report (and consequently receive the bonus) based on $\hat{x}_i=x_i$. Envelop theorem thus implies that the optimal choice of the parameter should satisfy
 \begin{equation}
   \frac{dV(B(R_i), \alpha_i)}{d x_i}=\frac{\partial V(B(R_i), \alpha_i)}{\partial x_i}.
   \label{I1}
  \end{equation}
We note that 
  \begin{align*}
& \frac{dV(B(R_i), \alpha_i)}{d x_i}\\
&= \frac{\partial V(B(R_i), \alpha_i)}{\partial B} \frac{\partial B(R(x_i))}{\partial x_i}+\frac{\partial V(B(R_i), \alpha_i)}{\partial R_i}  \frac{\partial R(x_i)}{\partial x_i}\\&\qquad\qquad\qquad\qquad+\frac{\partial V(B(R_i), \alpha_i)}{\partial x_i}  \frac{\partial x_i}{\partial x_i}\\   
&= \frac{\partial B(R_i)}{\partial x_i}-\beta_i(R_i-x_i)\frac{\partial R_i}{\partial x_i}+\frac{\partial V(B(R_i), \alpha_i)}{\partial x_i}.
\end{align*}
Thus, \eqref{I1} yields 
 \[
  \frac{\partial B(R_i)}{\partial x_i}= \beta_i(R_i-x_i)\frac{\partial R_i}{\partial x_i}.\]

To evaluate the second order condition, we start with the first order incentive compatibility condition as 
 \begin{equation*}
   \frac{dV(B(R_i), \alpha_i)}{d \hat{x}_i}\Bigg\vert_{\hat{x}_{i}=x_{i}}=0,
  \end{equation*}
  and differentiate both sides with respect to $x_{i}$ to obtain
  \[
  \frac{\partial^2 V(B(R_i), \alpha_i)}{\partial \hat{x}_i^2} \Bigg\vert_{\hat{x}_{i}=x_{i}}  \frac{\partial \hat{x}_i}{\partial x_i}\Bigg\vert_{\hat{x}_{i}=x_{i}}+ \frac{\partial^2 V(B(R_i), \alpha_i)}{\partial x_i \partial \hat{x}_i}\Bigg\vert_{\hat{x}_{i}=x_{i}}=0.
  \]
The second-order condition for the optimal choice of $\hat{x}_i$ implies that 
$$\frac{\partial^2 V(B(R_i), \alpha_i)}{\partial \hat{x}_i^2} \Bigg\vert_{\hat{x}_{i}=x_{i}}  \leq 0.$$ 
Thus, we can write  
\begin{align*}
&\frac{\partial^2 V(B(R_i), \alpha_i)}{\partial x_i \partial \hat{x}_i}\Bigg\vert_{\hat{x}_{i}=x_{i}}\geq 0\\
\Rightarrow& \frac{\beta_{i}}{2}\frac{\partial^2 (R_i-x_i)^{2}}{\partial x_i^2} \frac{\partial R(x_i)}{\partial x_i} \geq0  \\
\Rightarrow &\frac{\partial R(x_i)}{\partial x_i}\geq0.
\end{align*}
In particular, for the contract $B(R_i)=\mu \left(R_i-R_{0}\right)$, it is straightforward to see that these conditions reduce to \[\frac{\partial R_i}{\partial x_i}(\mu-\beta_i (R_i-x_i))=0, \quad \frac{\partial R_i}{\partial x_i}\geq0.\]

\end{proof}
\textbf{Proof of Theorem \ref{pro44}}

\begin{proof}
First, the optimal load reduction is specified by Theorem~\ref{pr0}. Thus, according to \eqref{Rstar2}, the optimal report is specified as
\begin{align*}
R^*_i&=x_{i}+\frac{1}{\beta_{i}}\frac{\partial \E_{\mathcal{E}_{-i}}[B_{i}(R_i)]}{ \partial R_i}\bigg\rvert_{R_{i}=R^*_{i}}\\
&=x_i+\frac{\mu}{\beta_i}.
\end{align*}
Using this report, we can calculate the optimal effort exerted by the customer using Theorem~\ref{pro22}. Thus, from~(\ref{effortstar}), we have
\begin{align*}
a_i^*=&\alpha_i+\frac{\partial \E_{\mathcal{E}}\left[B_{i}(R^*_i)-\frac{1}{2\beta_i}\left(\frac{\partial B_{i}(R_i)}{\partial R_i}\big\rvert_{R_{i}=R_{i}^{*}}\right)^2\right]}{\partial a_i}\Bigg\rvert_{a_{i}=a_{i}^{*}}\\
=&\alpha_i+\frac{\partial \E[\mu (R_i^*-R_0)-\frac{\mu^2}{2\beta_i}]}{\partial a_i}\\=&\alpha_i+\frac{\partial \E[\mu (x_i+\frac{\mu}{\beta_i}-R_0)-\frac{\mu^2}{2\beta_i}]}{\partial a_i}\\
=&\mu+\alpha_i.
\end{align*}
The optimal contract can now be specified using Theorem~\ref{proposition_opt_contract}. First,~\eqref{optalfa} implies the relation
\begin{align}
\nonumber \alpha_i^*&=1-\frac{a_i^*+\frac{\partial \E_{\mathcal{E}}[B_{i}(R_i)]}{\partial \alpha_i}\Big\rvert_{\alpha_{i}=\alpha_{i}^{*}}}{\frac{\partial a^*_i}{\partial \alpha_i}\Big\rvert_{\alpha_{i}=\alpha_{i}^{*}} }\\
\nonumber &=1-(\mu^*+\alpha^*_i)-\frac{\partial \E_{\mathcal{E}}[\mu^* (x_i+\frac{\mu^*}{\beta_i}-R_0)]}{\partial \alpha_i}\\
&=0.5-\mu^*.
\label{eq:temp_1}
\end{align}
On the other hand,~(\ref{bp3}) implies  the relation
\begin{align}
\nonumber \mu^*&= \argmax\left[ (1-\alpha^*_i)a_i^*-\E_{\mathcal{E}}[B^*(R_i^*)]\right]\\
 \nonumber&=\argmax\left[ (1-\alpha_i)(\mu+\alpha_i)-\E[\mu (x_i+\frac{\mu}{\beta_i}-R_0)]\right]\\
\nonumber  &=\argmax\left[(1-\alpha_i)(\mu+\alpha_i)-\mu (\mu+\alpha_i+\frac{\mu}{\beta_i}-R_0)\right]\\
&=\frac{0.5-\alpha^*+\frac{R_0}{2}}{1+\frac{1}{\beta}}.
\label{eq:temp_2}
\end{align}
From~(\ref{eq:temp_1}) and~(\ref{eq:temp_2}), we solve $\mu^*=\frac{c \beta_i}{2}.$ Finally, it is straight-forward to check that the optimal contract satisfies all the constraints in the problem.
\end{proof}

\textbf{Proof of Lemma \ref{pro551}}

\begin{proof}
With the proposed bonus function, the utility of customer $i$ depends on the true load reduction by the other customers since $V_i$   is a function of   $x_j$, where $j=1, \cdots, N$, $j\neq i$. Since Assumption \ref{asum4} states that customer $i$ does not have access to the load savings ${\mathcal{E}_{-i}}$ by other customers at the time of generating the report and obtaining the consequent bonus, the contract is incentive compatible if the expected utility of customer $i$ (with expectation taken with respect to ${\mathcal{E}_{-i}}$) is maximized if the customer $i$ calculates the report based on her true load saving $x_{i}$. Now, following the proof of Lemma~\ref{pro1}, we can obtain that necessary and sufficient conditions for incentive compatibility as\begin{align}
\label{second_incentive_first}
&\frac {\partial {\E_{\mathcal{E}_{-i}}\left[B(R_i, \sum_{j=1}^{N} R_j)\right]}}{\partial {x_i}}= \beta_i(R_i-x_i)\frac {\partial {R_i}}{\partial {x_i}}\\
\nonumber &\frac{\partial R_i}{\partial x_i}\geq0.
\end{align}
In particular, for the contract $B_i(\cdot)=R_i(\lambda-\sum _{j=1}^{N}R_j)$, it is straightforward to see that these conditions reduce to  
\begin{align*}
&\frac{\partial R_i}{\partial x_i}\left[\lambda+\beta_i x_i-(\beta_i+2)R_i-\sum\limits_{\substack{{j=1}\\{j\neq i}}}^{N} \E_{\mathcal{E}_{-i}}[R_j]\right]= R_i \frac{\partial \sum\limits_{\substack{{j=1}\\{j\neq i}}}^{N} \E_{\mathcal{E}_{-i}}[R_j]}{\partial x_i}, \\& \frac{\partial R_i}{\partial x_i}\geq0.
\end{align*}
\end{proof}

%
%
%

\textbf{Proof of Theorem \ref{theorem2}}
\begin{proof}
We first prove that the contract structure and the effort and the report specified in the theorem statement specifies a Nash equilibrium and then show that the equilibrium is unique and it always exists. To this end, we start by identifying the optimal report as specified by Theorem~\ref{pr0} with the specified bonus function and the assumption $\beta_{i}=\beta$. 

\textit{Proof of~(\ref{eq:report_shared}):} We have
\begin{align}
\nonumber R^*_i&=x_{i}+\frac{1}{\beta}\frac{\partial \E_{\mathcal{E}_{-i}}[B_{i}(R_i)]}{ \partial R_i}\bigg\rvert_{R_{i}=R^*_{i}}\\
\nonumber &=x_{i}+\frac{1}{\beta}\frac{\partial \E_{\mathcal{E}_{-i}}[R_i(\lambda-R_i-\sum\limits_{\substack{{j=1}\\{j\neq i}}}^{N}R_j)]}{ \partial R_i}\bigg\rvert_{R_{i}=R^*_{i}}\\
\nonumber &=x_{i}+\frac{1}{\beta}\frac{\partial\left(R_i(\lambda-R_i-\E_{\mathcal{E}_{-i}}[\sum\limits_{\substack{{j=1}\\{j\neq i}}}^{N}R_j])\right)}{ \partial R_i}\bigg\rvert_{R_{i}=R^*_{i}}\\
\nonumber &=x_{i}+\frac{1}{\beta}\left(\lambda-R_i^{*}-\sum\limits_{\substack{{j=1}\\{j\neq i}}}^{N}\E_{\mathcal{E}_{-i}}[R_j]- R_i^{*}\right)\\
\label{14}\Rightarrow R^{*}_{i} &=\frac{\lambda-\sum\limits_{\substack{{j=1}\\{j\neq i}}}^{N} \E_{\mathcal{E}_{-i}}[R_j]+\beta x_i}{\beta+2},
\end{align}
where we have used the fact that according to Assumption~\ref{asum4}, $R_{j}$ is a function of $x_{j}$ only and  $\sum\limits_{\substack{{j=1}\\{j\neq i}}}^{N}\E_{\mathcal{E}_{-i}}[R_j]$ is not a function of $R_i.$ We take the expectation of both sides of~(\ref{14}) with respect to $\mathcal{E}$ to obtain
\begin{align}
\nonumber\E_{\mathcal{E}}(R_i^*)&=\frac {\lambda-\sum\limits_{\substack{{j=1}\\{j\neq i}}}^{N} \E_{\mathcal{E}}(R_j^*)+\beta a_i}{\beta+2}\\
&=\frac {\lambda-\sum\limits_{j=1}^{N} \E_{\mathcal{E}}(R_j^*)+\beta a_i}{\beta+1}\label{ERN}\\
\Rightarrow \sum\limits_{j=1}^{N} \E_{\mathcal{E}}(R_j^*)&= \frac{N\lambda +\beta \sum\limits_{j=1}^{N} a_j }{\beta+1+N}.\label{ERB}
\end{align}

Subtracting \eqref{ERN} from \eqref{ERB} thus yields
\begin{align}
\nonumber
&\sum\limits_{\substack{{j=1}\\{j\neq i}}}^{N} \E_{\mathcal{E}}(R_j^*)\\
\nonumber&=\frac{\lambda(N-1)}{\beta+1+N}+\frac{\beta}{(\beta+1)}\left( \frac{(\beta+2)\sum\limits_{j=1}^{N} a_j-a_i(\beta+1+N) }{(\beta+1+N)}\right)\\ 
\nonumber&=\frac{\lambda(N-1)}{\beta+1+N}+\frac{\beta}{(\beta+1)}\left(\frac{(\beta+2)\sum\limits_{\substack{{j=1}\\{j\neq i}}}^{N} a_j-(N-1)a_i }{(\beta+1+N)}\right)\\ 
&=G-Fa_{i}.\label{1j} 
\end{align}
Substituting this value in \eqref{14}, we obtain~(\ref{eq:report_shared}).

\textit{Proof of~(\ref{eq:action_shared}):}
The optimal choice of effort $a^*_i$ is specified by Theorem~\ref{pro22}. With the given bonus function and the assumptions $\beta_{i}=\beta$ and $\alpha_{i}=\alpha/N,$ we obtain
\begin{align*}
a^*_i&=\argmax_{a_{i}} \E_{\mathcal{E}}\left[\frac{\alpha}{N} x_i+ R^*_i(\lambda-\sum\limits_{j=1}^{N} R^*_j)-\frac{\beta}{2}(R^*_i-x_i)^2-\frac{a_i^2}{2}\right]\\
&=\argmax_{a_{i}} \left(\frac{\alpha}{N} a_i+ \E_{\mathcal{E}}\left[R^*_i(\lambda-\sum\limits_{j=1}^{N} R^*_j)-\frac{\beta}{2}(R^*_i-x_i)^2\right]-\frac{a_i^2}{2}\right).
\end{align*}
Using the first derivative condition to evaluate the optimal choice of $a_i,$ we set 
\begin{align}
\nonumber a_i^{*}-\frac{\alpha}{N}&=\frac{\partial{\E_{\mathcal{E}} \left[R_i^*(\lambda-\sum\limits_{j=1}^{N} R^*_j)\right]}}{\partial a_i}\Bigg\vert_{a_{i}=a_i^{*}}\\&\qquad\qquad-\frac{\beta}{2}\frac{\partial{\E_{\mathcal{E}}\left[(R_i^*-x_i)^2\right]}}{\partial a_i}\Bigg\vert_{a_{i}=a_i^{*}}.
\label{EVV}
\end{align}
We evaluate the terms on the left hand side as follows.
\begin{align}
\nonumber&\frac{\partial{\E_{\mathcal{E}} \left[R_i^*(\lambda-\sum\limits_{j=1}^{N} R^*_j)\right]}}{\partial a_i}\\
\nonumber&=\lambda\frac{\partial{\E_{\mathcal{E}} \left[R_i^*\right]}}{\partial a_i}-\frac{\partial{\E_{\mathcal{E}} \left[\left(R_i^*\right)^{2}\right]}}{\partial a_i}-\frac{\partial{\E_{\mathcal{E}} \left[R_i^*\sum\limits_{\substack{{j=1}\\{j\neq i}}}^{N} R^*_j\right]}}{\partial a_i}\\
&=\lambda\frac{\partial{\E_{\mathcal{E}} \left[R_i^*\right]}}{\partial a_i}-\frac{\partial{\E_{\mathcal{E}} \left[\left(R_i^*\right)^{2}\right]}}{\partial a_i}-\frac{\partial{\E_{\mathcal{E}} \left[R_i^*\right]\E_{\mathcal{E}}\left[\sum\limits_{\substack{{j=1}\\{j\neq i}}}^{N} R^*_j\right]}}{\partial a_i}\label{eq:ind_expectations_1}\\
\nonumber&=\lambda\frac{\partial{\E_{\mathcal{E}} \left[\frac{\lambda-\left(G-F a_i\right)+\beta x_i}{\beta+2}\right]}}{\partial a_i}-\frac{\partial{\E_{\mathcal{E}} \left[\left(\frac{\lambda-\left(G-F a_i\right)+\beta x_i}{\beta+2}\right)^{2}\right]}}{\partial a_i}\\
&\qquad-\frac{\partial{\E_{\mathcal{E}} \left[\frac{\lambda-\left(G-F a_i\right)+\beta x_i}{\beta+2}\right]\left(G-Fa_{i}\right)}}{\partial a_i}
\label{eq:ind_expectations_2}\\
\nonumber&=\lambda\frac{F +\beta}{\beta+2}-\frac{2\left((F+\beta)(\lambda-G)+(F+\beta)^{2}a_i\right)}{\left(\beta+2\right)^{2}}\\
&\qquad+ \frac{\left(\lambda-G+\left(F +\beta\right)a_i\right)F+\left(G-Fa_{i}\right)\left(F +\beta\right)}{\beta+2}  
\label{eq:ind_expectations_3}
\end{align}
where~(\ref{eq:ind_expectations_1}) follows from Assumption \ref{assum11} and the fact that the report $R_{i}$ is not dependent on $x_{j}$ for all $j\neq i$, ~(\ref{eq:ind_expectations_2}) follows from~(\ref{eq:report_shared}) and~(\ref{1j}),~(\ref{eq:ind_expectations_3}) follows from Assumption~\ref{asum3} and straight-forward algebraic manipulation. Similar manipulation yields
\begin{equation}
\frac{\partial{\E_{\mathcal{E}}\left[(R_i^*-x_i)^2\right]}}{\partial a_i}=\frac{(2(F-2)(\lambda-G)-8F a_i+8a_i+ 2F^2 a_i)}{(\beta+2)^2}.
\label{eq:ind_expectations_4}
\end{equation}
Substituting \eqref{eq:ind_expectations_3} and \eqref{eq:ind_expectations_4} in \eqref{EVV} and solving for $a_{i}^{*}$ yields 
\begin{align}
 \label{p44} a_i^*&=\frac{\frac{\alpha}{N}+\frac{\beta+F}{\beta+2}(\lambda-G)}{1+\frac{2(\beta+F)^2}{(\beta+2)^2}-\frac{2(\beta+F)F}{(\beta_i+2)}+\frac{\beta(4-4F+F^2)}{(\beta+2)^2}}\\
\label{aopt12}&=\frac{\frac{\alpha}{N}+A\lambda-B\sum\limits_{\substack{{j=1}\\{j\neq i}}}^{N} a_j^* }{C}.
\end{align}
Summing  \eqref{aopt12} $N$ times for $i=1,\cdots,N,$ we obtain
\begin{align*}
(C+(N-1)B)\sum_{j= 1}^{N} a^*_j &=N(\frac{\alpha}{N}+A\lambda)\\
\Rightarrow \sum_{j= 1}^{N} a^*_j =\frac{N(\frac{\alpha}{N}+A\lambda)}{C+(N-1)B}&=\frac{\alpha+NA\lambda}{C+(N-1)B}.
\end{align*}
Substituting in \eqref{aopt12} finally yields~(\ref{eq:action_shared}). The second derivative condition implies that $a_{i}^*$ is indeed a maximizer.

\textit{Proof of~(\ref{eq:bonus_shared}):} Given~(\ref{eq:action_shared}) and Assumption~\ref{asum3}, we have that the overall load reduction
\begin{align*}
\Gamma&=\sum_{i=1}^{N}\E_{\mathcal{E}}\left[x_{i}\right]\\
&=Na_{i}^*\\
&=N\frac{\alpha_i^*+A\lambda^*}{C+(N-1)B}.
\end{align*}
A simple realignment yields~(\ref{eq:bonus_shared}).

\textit{Proof of~(\ref{eq:share_shared}):} Given~(\ref{eq:bonus_shared}), the optimal share is specified by~(\ref{optpay}). With the specified bonus function, we have
\begin{align}
\nonumber \alpha^*_i&= \argmax_{\alpha_{i}} \E_{\mathcal{E}}\left[\Pi\left(\{B_{i}(R_i^*)\}, \{\alpha_i^*\}\right)\right]\\
 &= \frac{1}{N}\argmax_{\alpha}\E_{\mathcal{E}}\left[(1-\frac{\alpha}{N})\sum_{i=1}^{N}x_i^*-\sum_{i=1}^{N}R^*_{i} (\lambda^{*}- \sum_{j=1}^{N}R^*_{j})\right]. 
\label{ps34}
\end{align}
From  \eqref{14} and \eqref{ERB}, we can write 
\begin{align}
\nonumber \sum_{i=1}^{N}R_{i}^*=&\frac{N\lambda^*-(N-1)\sum\limits_{j=1}^{N} \E\left[R_j^*\right]+\sum\limits_{j=1}^{N}\beta x_i}{\beta+2}\\ 
=&E\lambda^*+B_1\sum_{j= 1}^{N} a^*_j +C_1\sum_{j= 1}^{N} x_j,
\label{ps341}
\end{align}
where $E=\frac{N}{\beta+1+N}$, $C_1=\frac{\beta}{\beta+2}$, and $B_1=-\frac{(N-1)\beta}{(\beta+2)(\beta+1+N)}$.
Substituting~(\ref{ps341}) in \eqref{ps34}, we obtain
\begin{align}
\nonumber \alpha^*_i&=\frac{1}{N}\argmax_{\alpha}\Biggl((1-\frac{\alpha}{N}) \sum_{j= 1}^{N} a^*_j \\
\nonumber&- \E_{\mathcal{E}}\Biggl[ \left(E\lambda^*+B_1\sum_{j= 1}^{N} a^*_j +C_1\sum_{j= 1}^{N} x_j\right)\\&\qquad\qquad\left((1-E)\lambda^*-B_1\sum_{j= 1}^{N} a^*_j-C_1\sum_{j= 1}^{N} x_j\right)\Biggr]\Biggr).
\label{w95}
\end{align}
The first order derivative condition implies that $\alpha^*$ is given by the equation
\begin{multline*}
-\frac{\sum_{j= 1}^{N} a^*_j}{N}+\frac{1}{C+(N-1)B}\Bigl[(1-\frac{\alpha^*}{N})\\+(B_1+C_1)\left(-\lambda^*+2(E\lambda^* (B_1+C_1)\sum_{j= 1}^{N} a^*_j)\right)\Bigr]=0,
\end{multline*}
which yields
\begin{align}
\nonumber \alpha_{i}^*&=\frac{1-\lambda^*[A(1-2ND)+(B_1+C_1)(1-2E)]}{2(1-ND)}\\ \nonumber
&=\frac{1-\lambda^*[A(1-2ND)+\frac{\beta(1-2E)}{\beta+1+N}]}{2(1-ND)}.
\end{align}

\textit{Proof of optimality of the contract:} $\alpha_{i}^*$ and $\lambda_{i}^*$ have been chosen to satisfy~(\ref{optpay}) and the constraint on the total load reduction. It is easy to verify that the individual rationality and incentive compatibility constraints are met. Thus, the contract is optimal in the sense of solving Problem $\mathcal{P}_{3}$. That the Nash equilibrium always exists and is unique is clear from the above derivation of the contract and the optimal actions and reports.

 \end{proof}

\bibliographystyle{IEEEtran}
%

\begin{IEEEbiography}{Donya Ghavidel-Dobhakhshari}
received her B.S. degree in Electrical Engineering at Iran University of Science and Technology, Tehran, Iran, 2014, and her M.S. degree in Electrical Engineering from the University of Notre Dame, Notre Dame, IN, USA, in 2016, where she is currently pursuing the Ph.D. degree in Electrical Engineering.
Her research interests lie in game theory, mechanism design, and their applications to power systems and networks.
\end{IEEEbiography}

\begin{IEEEbiography}{Vijay Gupta}
Vijay Gupta is a Professor in the Department of Electrical Engineering at the University of Notre Dame, having joined the faculty in January 2008. He received his B. Tech degree at Indian Institute of Technology, Delhi, and his M.S. and Ph.D. at California Institute of Technology, all in Electrical Engineering. Prior to joining Notre Dame, he also served as a research associate in the Institute for Systems Research at the University of Maryland, College Park. He received the 2013 Donald P. Eckman Award from the American Automatic Control Council and a 2009 National Science Foundation (NSF) CAREER Award. His research and teaching interests are broadly at the interface of communication, control, distributed computation, and human decision making.
\end{IEEEbiography}

\end{document}